\documentclass[11pt]{article}

\usepackage{bbm}
\usepackage{amsthm}
\usepackage{amsmath,amssymb}
\usepackage{comment}
\usepackage[dvipdfmx]{graphicx}
\usepackage{cases}

\theoremstyle{plain}% Theorem-like structures provided by amsthm.sty
\newtheorem{theorem}{Theorem}[section]
\newtheorem{lemma}[theorem]{Lemma}

\newtheorem{proposition}[theorem]{Proposition}

\theoremstyle{definition}
\newtheorem{definition}[theorem]{Definition}
\newtheorem{example}[theorem]{Example}

\theoremstyle{remark}
\newtheorem{remark}{Remark}

\newcommand{\head}[0]{\mathfrak{h}}
\newcommand{\tail}[0]{\mathfrak{t}}
\newcommand{\inv}[0]{^{-1} }
\newcommand{\arc}[0]{\mathcal{A}}

\newcommand{\mat}[0]{\mbox{\rm Mat}}
\newcommand{\trace}[0]{\mbox{\rm tr}}

\title {The Ihara expression of a generalization \\ of the weighted zeta function on a finite digraph}
\author{Ayaka Ishikawa \\
Faculty of Engineering \\
Yokohama National University \\ 
Hodogaya, Yokohama, 240-8501, JAPAN
}
\date{}

\begin{document}
\maketitle

\begin{abstract}
We define a new weighted zeta function for a finite digraph
and obtain its determinant expression called the Ihara expression.
The graph zeta function is a generalization of the weighted graph zeta function introduced in previous research.
That is, our result makes it possible to derive the Ihara expressions of the previous graph zeta functions for any finite digraphs.
\end{abstract}

%%%%%%%%%%%%%%%%%%%%%%%%%%%%%%%%%%%%%%%%%%%%%%%%
\section{Introduction}

A graph zeta function is an analogue of the Selberg zeta function for a graph.
Generally, it is defined as an exponential of a generating function or an Euler product, called the exponential expression and the Euler expression, respectively.
%These expressions count the number of closed paths on a graph or the sum of their weights.
A graph zeta function also has a determinant expression called the Hashimoto expression written by an edge matrix of a graph.
If a graph zeta function satisfies some conditions,  particularly the adjacency condition, then the graph zeta function always has the three expressions \cite{morita2020ruelle}.
The Ihara expression is the determinant expression written by the weighted adjacency matrix and the weighted degree matrix.
Although the conditions under which a graph zeta function has the Ihara expression are unknown,
it has already been confirmed that the generalized weighted zeta function has the Ihara expression \cite{IDE2021227, ishikawa21}.
Since most graph zeta functions in previous studies are special cases of the generalized weighted zeta function,
the Ihara expression of the generalized weighted zeta function has been recognized as the general form in the Ihara expressions.

%The Ihara expression is useful to derive the eigenvalues of transition matrices of quantum walks.
The Ihara expression helps derive the eigenvalues of quantum walk transition matrices.
A quantum walk is the quantum version of a random walk studied in various fields:
 quantum algorithm, financial engineering, and laser isotope separation, for example (see, e.g., \cite{matsuoka2011theoretical, Orrell_2020, Portugal_2018}).
Some behavior of a walker, such as periodicity and localization, are derived using the eigenvalues of the transition matrix.
In particular, the Sato zeta function \cite{sato2007new} gives the characteristic polynomial of the transition matrix of the Grover walk \cite{grover1966fast}.
Moreover, Konno and Sato obtained the spectrum of the Grover transition matrix by transforming the Sato zeta function as the Ihara expression \cite{konno2012relation}.
The Grover walk has a generalized model called the Szegedy walk \cite{szegedy2004quantum}.
The graph zeta function that can give the eigenfunction of the Szegedy transition matrix was introduced by 
Ishikawa and Konno \cite{ishikawa22}.
Although it is a generalization of the Sato zeta function, it is not a particular case of the generalized weighted zeta function.
%Obviously, the graph zeta function is not a special case of the generalized weighted zeta function.
For a symmetric digraph corresponding to a graph,
the Ihara expression is the same as that for the generalized weighted zeta function.
However, for a general digraph, the Ihara expression differs from any previous graph zeta functions.
%However, for a general digraph, the Ihara expression is different from any previous graph zeta functions.
%different from the Ihara expression of the generalized weighted zeta function.
%The result shows that the Ihara expression of the generalized weighted zeta function is not a general form.

In this paper, we define a new graph zeta function.
It is a generalization of the Sato zeta function and Ishikawa-Konno's zeta function.
%In this paper, we define the graph zeta function{\color{red},} which is a generalization of the Sato zeta function and Ishikawa and Konno's zeta function.
We also derive the Ihara expression with a standard form regardless of whether the graph is symmetric.
%We also derive the Ihara expression that has a common form regardless of whether the graph is symmetric or not.
The Ihara expression is a general form of the Ihara expression of a graph zeta function.

The rest of the paper is organized as follows.
Section \ref{sec:pre} defines notations associated with graphs and our graph zeta function.
%In Section \ref{sec:pre}, we define notations associated with graphs and our graph zeta function.
The zeta function is introduced as an exponential expression, and we confirm that the exponential expression equals the Euler and the Hashimoto expressions.
%The zeta function is introduced as {\color{red}an} exponential expression, and we confirm that the exponential expression equals the Euler and the Hashimoto expressions.
We show the Ihara expression in Section \ref{main}, which is the main theorem of our paper.
%The Ihara expression is shown in Section \ref{main} which is the main theorem of our paper.
%In Section \ref{sec: exam}, we give two examples on a graph and a digraph.

Throughout this paper, 
graphs (resp. digraphs) are finite, and multi-edge (resp. multi-arcs) and multi-loops are allowed.
We use the following symbols.
For positive integers $m$ and $n$, let $\mat(m,n;\mathbb{C})$ be the set of $m\times n$ matrices over $\mathbb{C}$.
%For a matrix $M\in\mat(n,n;\mathbb{C})$,
%let ${\rm Spec}(M)$ be the spectrum of $M$.
An $n$-square matrix with all one denotes by $\mathbbm{1}_{n}$.
%In particular, if $m=n$, we write $\mathbbm{1}_{m\times n}$ as $\mathbbm{1}_{n}$.
For a proposition $P$,
we define $\delta_{P}$ as follows: $\delta_{P}=1$ if $P$ is true, $\delta_{P}=0$ if $P$ is false.
Let $\overline{\delta}_{uv}$ be a function such that $\overline{\delta}_{uv}=1$ if $u\neq v$, $\overline{\delta}_{uv}=0$ otherwise.
%For the Kronecker delta $\delta_{uv}$,
%let $\overline{\delta}_{uv}=1$ if $u\neq v$, $\overline{\delta}_{uv}=0$ otherwise.

%%%%%%%%%%%%%%%%%%%%%%%%%%%%%%%%%%%%%%%%%%%%%%%%
%%%%%%%%%%%%%%%%%%   PRELIMINARY   %%%%%%%%%%%%%%%%%%%%
%%%%%%%%%%%%%%%%%%%%%%%%%%%%%%%%%%%%%%%%%%%%%%%%
\section{Preliminary}\label{sec:pre}
\noindent
A {\it graph} $G=(V,E)$ is a pair of {\it vertex} set $V$ and {\it edge} set $E$.
	The edge set $E$ is a multiset of $2$-multisubsets of $V$.
	%where $E$ consists of $2$-subsets of $V$.
If both $V$ and $E$ are finite, then $G$ is called {\it finite}.
We call an edge $e=\{v,v\}$ a loop.
The number $\deg(u) := \#\{ \{u,v\} \in E | v\in V \}$ is called the {\it degree} of $u$.
If there is at most one edge between every two vertices and there are no loops,
then the graph is called {\it simple}.
	Let $\arc$ be a multiset of ordered pairs of two vertices (possibly same).
%	Let $V$ be a vertex set and $\arc$ a set of ordered pairs of two vertices.
We call the pair $\Delta =(V,\arc)$ a {\it digraph} and an element of $\arc$ an {\it arc}.
For an arc $a=(u,v)$, 
$u$ and $v$ are called the {\it tail} and the {\it head} of $a$ denoted by $\tail(a)$ and $\head(a)$, respectively.
For two vertices $u,v\in V$,  
let $\arc_{uv}, \arc_{u*}, \arc_{*v},\arc(u,v)$ be the subsets of arc set such that
$\arc_{uv}:= \{ a\in\arc \ | \ \tail(a)=u, \head(a)=v \}$, 
$\arc_{u*} := \{ a\in\arc | \tail(a)=u\}$, 
$\arc_{*v} := \{ a\in\arc \ | \ \head(a)=v\}$
and $\arc(u,v) := \arc_{uv}\cup \arc_{vu}$.
We fix a total order $<$ on $V$, and if one writes $\arc(u,v)$, then the condition $u<v$ is always assumed.
For example, in Figure \ref{fig:digraph1}, we have
$\arc_{v_1,v_2}= \{ a_{21},a_{22} \}$ and $\arc_{v_1,v_2}= \{ a_{23},a_{24} \}$.
Thus, $\arc(v_1,v_2)=\{ a_{21},a_{22}, a_{23},a_{24} \}$ holds.

For a graph $G=(V,E)$, let $\arc(G):=\{  a_e=(u,v), a_e'=(v,u) | e=\{ v,u \}\in E \}$, and then the digraph $\Delta(G)=(V,\arc(G))$ is called the {\it symmetric digraph} of $G$.
	For the arcs $a_e, a_e'\in\arc(G)$ corresponding to an edge $e\in E$, let $\overline{a_e}$ (resp. $\overline{a_e'}$) denote $a_e'$ (resp. $a_e$).
%For an arc $a\in\arc(G)$ on the symmetric digraph $\Delta(G)=(V,\arc(G))$, 
%we denote by $\overline{a}$ the arc induced by the same edge as $a$.

On a digraph $\Delta$, for an arc $a$, let $a\inv$ denote the set of {\it inverse} arcs of $a$.
We define $a\inv$ as $a\inv := \{ \overline{a} \}$ if $\Delta$ is the symmetric digraph of a graph, otherwise $a\inv :=\arc_{\head(a)\tail(a)}$.
Note that for a loop $a\in\arc_{vv}$ on a digraph $\Delta$ with $\Delta\neq\Delta(G)$, we have $a\inv =\arc_{vv}$, and $a\inv$ includes $a$ itself.

Let $\Phi_\Delta$ be the set $\{ (u,v) \in V\times V \ | \ u\le v, \arc(u,v)\neq \emptyset \}$.
For an arc $a\in \arc_{uv}$, let $\arc[a]$ denote the set of arcs that have inverse arcs in common with $a$.
That is, any two arcs $a',a''\in\arc[a]$ satisfy ${a'}\inv = {a''}\inv$.
One can see that $\arc[a]=\{ a \}$ always holds for any arc $a$ if $\Delta=\Delta(G)$.
On the other hand, if $\Delta\neq \Delta(G)$, then $\arc[a]=\arc_{\tail(a)\head(a)}$ holds.
We assume that $\arc_{uv}\neq \emptyset$.
%and there is the set 
%$B_{uv}:=\{ a_1,\ldots ,a_s | a_i\inv \cap a_j\inv = \emptyset \ {\rm if} \ i\neq j \}\subset \arc_{uv}$ satisfying 
Let $B_{uv}$ be the subset of $\arc_{uv}$ satisfying 
$a\inv \cap {a'}\inv =\emptyset$ for any $a,a'\in B_{uv}$ 
and $\arc_{uv}=\sqcup_{a\in B_{uv}} \arc[a]$.
Then, $\arc_{vu}$ is written by $\arc_{vu}=\sqcup_{a\in B_{uv}} a\inv$.
For $a\in\arc$, let $\arc(a)$ denote the set $\arc[a]\cup a\inv$.
One can see that $\arc(u,v)$ is written by $\sqcup_{a\in B_{uv}}\arc(a)$.
For each $(u,v)\in\Phi_\Delta$, let $B(u,v):=B_{uv}$ if $\arc_{uv}\neq \emptyset$, $B(u,v):=B_{vu}$ otherwise.
Then, the following holds:
$$
	\arc = \sqcup_{(u,v)\in\Phi_\Delta} \arc(u,v) 
	= \sqcup_{(u,v)\in\Phi_\Delta} \sqcup_{a\in B(u,v)} \arc(a).
$$
\begin{figure}[h]
\begin{minipage}[b]{0.49\columnwidth}
\center
\includegraphics[width=150pt]{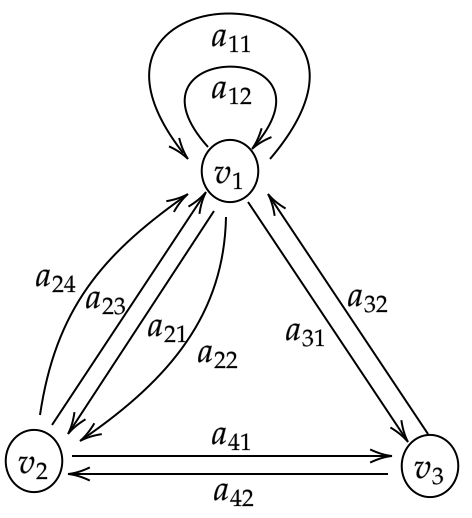}
\label{fig:digraph1}
\caption{$\Delta=(V,\arc)$}
\end{minipage}
\begin{minipage}[b]{0.49\columnwidth}
\center
\includegraphics[width=150pt]{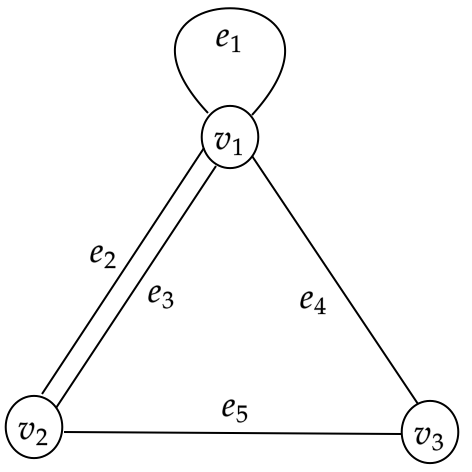}
\label{fig:graph1}
\caption{$G=(V,E)$}
\end{minipage}
\end{figure}

Let us explain the symbols with an example.
\begin{example} \label{exm:1}
We consider the digraph $\Delta$ in Figure \ref{fig:digraph1}.
Let assume a total order $<$ on $V$ as $v_1<v_2<v_3$.
We have $\Phi_\Delta =\{ (v_1,v_1), (v_1,v_2), (v_1,v_3), (v_2,v_3)\}$.
If we regard $\Delta$ as a usual digraph, not a symmetric digraph of any graph,
then $a\inv = \arc_{vu}$ holds for any arc $a\in_{uv}$.
We have 
$$
	\arc[a_{11}]=\arc[a_{12}]=\arc_{v_1 v_1}, \quad \arc[a_{21}]=\arc[a_{22}]=\arc_{v_1v_2},% \ \mbox{ etc.}
$$ etc.
It follows that
\begin{align*}
	&\arc(a_{11})=\arc(a_{12})=\arc[a_{11}]\cup a_{11}\inv =\arc(v_1,v_1),\\
	&\arc(a_{21})=\arc(a_{22})=\arc[a_{21}]\cup a_{21}\inv =\arc_{v_1 v_2}\cup \arc_{v_2 v_1}=\arc(v_1,v_2), \\
	&\arc(a_{31})=\arc[a_{31}]\cup a_{31}\inv =\arc_{v_1v_3}\cup\arc_{v_3v_1}=\arc(v_1,v_3), \\
	&\arc(a_{41})=\arc[a_{41}]\cup a_{41}\inv =\arc_{v_2v_3}\cup\arc_{v_3v_2}=\arc(v_2,v_3).
\end{align*}
Since $\arc_{v_iv_j}\neq \emptyset$ for each $(v_i,v_j)\in\Phi_\Delta$, $B(v_i,v_j)=B_{v_iv_j}$ holds.
We assume that 
$$
	B_{v_1v_1}=\{ a_{11} \}, \quad B_{v_1v_2}=\{ a_{21} \}, \quad B_{v_1v_3}=\{ a_{31} \}, \quad B_{v_2v_3}=\{ a_{41} \}.
$$
The arc set $\arc$ is partitioned as follows:
\begin{align*}
	\arc &= \arc(v_1,v_1)\sqcup\arc(v_1,v_2)\sqcup\arc(v_1,v_3)\sqcup\arc(v_2,v_3) \\
	&= \arc(a_{11})\sqcup\arc(a_{21})\sqcup\arc(a_{31})\sqcup\arc(a_{41}).
\end{align*}

\end{example}

\begin{example}\label{exm:2}
On the other hand, we can regard $\Delta$ as a symmetric digraph of $G$ in Figure \ref{fig:graph1}.
In particular, we assign $e_2$ to $\{ a_{21},a_{23} \}$ and $e_3$ to $\{ a_{22},a_{24} \}$.
%That is, $a_{23}$ and $a_{24}$ are defined by $\overline{a_{21}}$ and $\overline{a_{22}}$, respectively.
It follows that
$
	B_{v_1v_2}=\{ a_{21},a_{22} \}
$
and the other $B_{v_iv_j}$ are the same as above.
Thus, the arc set $\arc$ is given by
$$
	\arc = \arc(a_{11})\sqcup\arc(a_{21})\sqcup\arc(a_{22})\sqcup\arc(a_{31})\sqcup\arc(a_{41}).
$$
\end{example}

A sequence of arcs $p=(a_i)_{i=1}^{k}$ is a {\it path} 
if it satisfies $\head(a_i) = \tail(a_{i+1})$ for each $i=1,2,\ldots, k-1$.
The number $k$, called the {\it length} of $p$, is denote by $|p|$.
If $\head(a_k)=\tail(a_1)$, then the path $p$ is called {\it closed}.
Let $X_k$ denote the set of closed paths of length $k$.
For $C\in X_k$,
we denote by $C^n$ the closed path that connects $C$ $n$ times.
It is called the $n$-th power of $C$.
If $C$ cannot be expressed as a power of a closed path shorter than $C$, then it is called {\it prime}.
For $C=(c_i)_{i=1}^k, C'=(c_i')_{i=1}^k \in X_k$, 
if there exists an integer $n$ such that $c_i=c'_{i+n}$ for any $i$, where the indices are taken modulo $k$,
then we denote the relation by $C\sim C'$.
The relation is an equivalence relation.
%Clearly, the relation $\sim$ is an equivalence relation.
An equivalence class is called a {\it cycle}, and we denote by $[C]$ the equivalence class of a closed path $C$.
Since any closed paths in $[C]$ have the same length, 
we define the {\it length} of $[C]$ to be the length of a closed path in $[C]$.
We denote by $|C|$ the length of $[C]$.
A cycle is {\it prime} if a closed path in the cycle is prime.
We denote the set of prime cycles by $\mathcal{P}$.
% We denote by $\mathcal{P}$ the set of prime cycles.

%%%%%%%%%%%%%%%  NEW SATO ZETA FUNCTION  %%%%%%%%%%%%%%%
%%%%%%%%%%%%%%%%%%%%%%%%%%%%%%%%%%%%%%%%%%%%%%%%

\subsection{A new graph zeta function}

Let $\Delta=(V,\arc)$ be a digraph.
For any maps $\tau_1, \tau_2, \upsilon_1, \upsilon_2 : \arc \to \mathbb{C}$,
let $\tau$ and $\upsilon$ be maps $\arc\times \arc \to \mathbb{C}$ defined by $\tau(a,a') := \tau_1(a)\tau_2(a')$ and 
$\upsilon(a,a')=\upsilon_1(a)\upsilon_2(a')$.
We define a map $\theta : \arc\times \arc \to \mathbb{C}$ as
\begin{align} \label{def:weight}
	\theta(a,a') := \tau(a,a')\delta_{\head(a)\tail(a')} - \upsilon(a,a')\delta_{a'\in a\inv}, 
\end{align}
where 
$\delta_{\head(a)\tail(a')}$ is the Kronecker delta.
For a closed path $C=(c_i)_{i=1}^k\in X_k$, 
let ${\rm circ}_\theta (C)$ denote the circular product 
$
	\theta(c_1,c_2) \theta(c_2,c_3) \ldots \theta(c_k,c_1).
$
Note that ${\rm circ}_\theta(C)={\rm circ}_\theta(C')$ holds if $C\sim C'$.
Let $N_k({\rm circ}_\theta) := \sum_{C\in X_k} {\rm circ}_\theta(C)$.

\begin{definition}
A graph zeta function for $\Delta$ is the following formal power series:
	\begin{align}\label{ident:zeta}
		Z_\Delta(t;\theta) := \exp \left(\sum_{k\ge 1}\frac{N_k({\rm circ}_\theta)}{k}t^k \right).
	\end{align}
\end{definition}

The map $\theta$ is called the {\it weight} of $Z_\Delta(t;\theta)$,
and the formal power series expression is called the {\it exponential expression} \cite{morita2020ruelle}.
Let 
\begin{align*}
	E_\Delta (t;\theta) := \prod_{[C]\in\mathcal{P}} \frac{1}{1-{\rm circ}_\theta (C)t^{|C|}}, \qquad 
	H_\Delta (t;\theta) := \frac{1}{\det(I-tM_\theta)},
\end{align*}
where $M_\theta=(\theta(a,a'))_{a,a'\in\arc}$.
The expressions $E_\Delta(t;\theta)$ and $H_\Delta(t;\theta)$ are called
the {\it Euler expression} and the {\it Hashimoto expression},
respectively (cf. \cite{morita2020ruelle}).

\begin{proposition}
	If $\theta : \arc\times\arc \to \mathbb{C}$ satisfies the condition 
	$$
		\theta(a,a') \neq 0 \Rightarrow \head(a)=\tail(a'),
	$$
	then $Z_\Delta(t;\theta)= E_\Delta (t;\theta)=H_\Delta (t;\theta)$.
\end{proposition}

\begin{proof}
	See \cite{morita2020ruelle}.
\end{proof}

The above condition for $\theta$ is called the {\it adjacency condition} \cite{morita2020ruelle}.
%Let $\Delta=(V,\arc)$ be a digraph.
%For any maps $\tau_1, \tau_2, \upsilon_1, \upsilon_2 : \arc \to \mathbb{C}$,
%let $\tau$ and $\upsilon$ be maps $\arc\times \arc \to \mathbb{C}$ defined by $\tau(a,a') := \tau_1(a)\tau_2(a')$ and 
%$\upsilon(a,a')=\upsilon_1(a)\upsilon_2(a')$.
%We define a map $\theta : \arc\times \arc \to \mathbb{C}$ as
%\begin{align} \label{def:weight}
%	\theta(a,a') := \tau(a,a')\delta_{\head(a)\tail(a')} - \delta_{a'\in a\inv}, 
%\end{align}
%where 
%$\delta_{\head(a)\tail(a')}$ is the Kronecker delta.
If $\theta(a,a')\neq 0$ for $a,a'\in\arc$, 
then $\delta_{\head(a) \tail(a')} =1$ holds.
Thus, the weight $\theta$ satisfies the adjacency condition,
and 
we can see $Z_\Delta(t;\theta)= E_\Delta (t;\theta)=H_\Delta (t;\theta)$.

\begin{remark}
	We assume that $\tau_1(a)=\upsilon_1(a)=1$ for any $a\in\arc$,
	then $Z_\Delta(t;\theta)$ is the generalized weighted zeta function \cite{morita2020ruelle}.
\end{remark}

%%%%%%%%%%%%%%%%%%%%%%%%%%%%%%%%%%%%%%%%%%%%%%%%
%%%%%%%%%%%%%%%%%%   MAIN THEOREM  %%%%%%%%%%%%%%%%%%%
%%%%%%%%%%%%%%%%%%%%%%%%%%%%%%%%%%%%%%%%%%%%%%%%
\section{Main theorem} \label{main}

\subsection{Auxiliary results}

Before describing the main theorem, we will show some lemmas.

\begin{lemma}\label{lem:I+tJ(u,u)}
	Let $M\in \mat(k,k;\mathbb{C})$ whose the $(i,j)$-entry is $m_1(i)m_2(j)$
	and $\mu := \trace (M)$.
	For a variable $t$, $(I+tM)\inv $ is written by
	\begin{eqnarray*}
		(I+tM)\inv = I-t(1-\mu^2t^2)\inv M +t^2(1-\mu^2t^2)\inv M^2,
%		=I-t(1+\mu t)\inv M,
	\end{eqnarray*}
	and $\det(I+tM) = 1+\mu t$ holds.
\end{lemma}

\begin{proof}
For a matrix $M$, we denote by $M_{ij}$ the $(i,j)$-entry of $M$.
	The $(i,j)$-entry of $M^2$ is
	%\begin{align*}
		$$
		(M^2)_{i,j} = \sum_{s=1}^k m_1(i)m_2(s) m_1(s)m_2(j) % \\
		%&= \left( \sum_{s=1}^k m_1(s)m_2(s) \right) m_1(i)m_2(j)\\
		= \mu m_1(i)m_2(j),
		$$
	%\end{align*}
	and we get $M^2= \mu M$.
	Thus, $M^n = \mu^{n-1} M$ holds, and a power series expansion of $(I+tM)\inv$ is as follows:
	\begin{align*}
		(I+tM)\inv &= I +\sum_{n\ge 1}(-tM)^n \\
			&= I + \sum_{n\ge 0}(-tM)^{2n+1} + \sum_{n\ge 1}(-tM)^{2n} \\
			&= I -t \sum_{n\ge 0}(\mu t)^{2n}M + t^2\sum_{n\ge 1}(-\mu t)^{2(n-1)}M^2 \\
			&= I-t(1-\mu^2t^2)\inv M +t^2(1-\mu^2t^2)\inv M^2.
		%&= I-t\sum_{n\ge 1}(-t)^{n-1} \mu^{n-1}M \\
		%&= I-t(1+\mu t)\inv M.
	\end{align*}

Let $D_1$ and $D_2$ be $k{\times}k$ diagonal matrices with $(D_1)_{ii}=m_1(i)$ and $(D_2)_{ii}=m_2(i)$, respectively.
%'' or ``Let $D_j=\operatorname{diag}\left(m_j(1),m_j(2),\cdots,m_j(k)\right)$ for $j=1,2$
%	Let $D_1:=(m_1(i)\delta_{ij})_{k\times k}$ and $D_2:=(m_2(i)\delta_{ij})_{k\times k}$.
	Since one can see that $M=D_1 \mathbbm{1}_{k} D_2$ holds,
	the determinant $\det(I+tM)$ is given by
	\begin{align*}
		\det(I+tM) 
		&= \det(I+tD_1 \mathbbm{1}_{k} D_2) = \det(I+tD_2D_1 \mathbbm{1}_{k}) \\
		&=
		\begin{vmatrix}
			1+m_1(1)m_2(1) t   &  m_1(1)m_2(1) t &  m_1(1)m_2(1) t & \ldots \\
			m_1(2)m_2(2) t & 1+ m_1(2)m_2(2) t  &  m_1(2)m_2(2) t & \ldots \\
			m_1(3)m_2(3) t &  m_1(3)m_2(3) t & 1+ m_1(3)m_2(3) t & \ldots \\
			\vdots & \vdots &\vdots  & \ddots
		\end{vmatrix}\\
		&=
		\begin{vmatrix}
			1+m_1(1)m_2(1) t   &  -1 &  -1 & \ldots \\
			m_1(2)m_2(2) t & 1  &  0 & \ldots \\
			m_1(3)m_2(3) t &  0 & 1 & \ldots \\
			\vdots & \vdots &\vdots  & \ddots
		\end{vmatrix}\\
		&=
		\begin{vmatrix}
			1+\mu t   &  0 &  0 & \ldots \\
			m_1(2)m_2(2) t & 1  &  0 & \ldots \\
			m_1(3)m_2(3) t &  0 & 1 & \ldots \\
			\vdots & \vdots &\vdots  & \ddots
		\end{vmatrix}\\
		&= 1+\mu t
	\end{align*}
	
\end{proof}

\begin{remark}
A more refined expression of $(I+tM)\inv$ is possible; however, we give the above expression for convenience.
\end{remark}

\begin{lemma}\label{lem:I+tJ(u,v)}
Let $M_1{\;\in\;}\operatorname{Mat}(k,l;\mathbb{C})$ and $M_2{\;\in\;}\operatorname{Mat}(l,k;\mathbb{C})$ with $(M_1)_{ij}=m_1(i)m_2(k+j)$ and $(M_2)_{ij}=m_1(k+i)m_2(j)$, respectively.
%	Let $M_1:=(m_1(i)m_2(k+j))\in\mat(k,l;\mathbb{C}), \ M_2:=(m_1(k+i)m_2(j))\in\mat(l,k;\mathbb{C})$
%	and
%	$
%		M := \begin{bmatrix}
%			O_k & M_1 \\
%			M_2 & O_l
%		\end{bmatrix} \in \mat(k+l,k+l;\mathbb{C})
%	$.
	Let $\mu_1$ and $\mu_2$ denote $\sum_{i=1}^k m_1(i)m_2(i)$ and $\sum_{i=k+1}^l m_1(i)m_2(i)$, respectively.
	For a variable $t$ and a matrix
	$
		M := \begin{bmatrix}
			O_k & M_1 \\
			M_2 & O_l
		\end{bmatrix}
	$, the following identity holds:
	\begin{eqnarray*}
		(I+tM)\inv = I-t(1-\mu_1\mu_2 t^2)\inv M +t^2(1-\mu_1\mu_2 t^2)\inv M^2,
	\end{eqnarray*}
	and the determinant $\det(I+tM)$ equals $1-\mu_1\mu_2 t^2$.
\end{lemma}

\begin{proof}
Let $M_1'{\;\in\;}\operatorname{Mat}(k,k;\mathbb{C})$ and $M_2'{\;\in\;}\operatorname{Mat}(l,l;\mathbb{C})$ with $(M_1)_{ij}=m_1(i)m_2(j)$ and $(M_2)_{ij}=m_1(k+i)m_2(k+j)$, respectively.
%	Let $M_1':=(m_1(i)m_2(j))\in\mat(k,k;\mathbb{C}), \ M_2':=(m_1(k+i)m_2(k+j))\in\mat(l,l;\mathbb{C})$.
	Since $M_1M_2 = \mu_2 M_1'$ and $M_2M_1=\mu_1M_2'$ hold,
	we get 
	$$
	 M^2=\begin{bmatrix}
			M_1M_2 & O \\
			O & M_2M_1
		\end{bmatrix}
		=\begin{bmatrix}
			\mu_2M_1' & O \\
			O & \mu_1M_2'
		\end{bmatrix}.
	$$
	It also holds that $(M_1')^n=\mu_1^{n-1}M_1', \ (M_2')^n=\mu_2^{n-1}M_2'$.
	Thus, $M^{2n}$ is written by
	\begin{align*}
	M^{2n}&=\begin{bmatrix}
			\mu_2^n (M_1')^n & O \\
			O & \mu_1^n (M_2')^n
		\end{bmatrix} \\
		& =(\mu_1\mu_2)^{n-1}\begin{bmatrix}
			\mu_2 M_1' & O \\
			O & \mu_1 M_2'
		\end{bmatrix} \\
		&=(\mu_1\mu_2)^{n-1}M^2.
	\end{align*}
	In addition, we have $M_1'M_1 = \mu_1 M_1$ and $M_2'M_2 = \mu_2 M_2$,
	and $M^{2n+1}$ is written by
	\begin{align*}
		M^{2n+1} &=(\mu_1\mu_2)^{n-1}
		\begin{bmatrix}
			O_k & \mu_2 M_1'M_1 \\
			\mu_1 M_2'M_2 & O_l
		\end{bmatrix} \\
		&=(\mu_1\mu_2)^{n-1}
		\begin{bmatrix}
			O_k & \mu_2 \mu_1M_1 \\
			\mu_1 \mu_2 M_2 & O_l
		\end{bmatrix} \\
		&=(\mu_1\mu_2)^n M
	\end{align*}
	Therefore, a power series expansion of $(I+tM)\inv$ is as follows:
	\begin{align*}
		(I+tM)\inv &= I +\sum_{n\ge 0}(-tM)^{2n+1} +\sum_{n\ge 1}(-tM)^{2n}\\
		 &= I -t \sum_{n\ge 0}t^{2n} (\mu_1\mu_2)^n M +t^2\sum_{n\ge 1}t^{2(n-1)}(\mu_1\mu_2)^{n-1}M^2\\
		 &= I-t (1-\mu_1\mu_2 t^2)\inv M +t^2(1-\mu_1\mu_2 t^2)\inv M^2.
	\end{align*}
	Since the matrix $I+tM$ is decomposed as
	$
	\begin{bmatrix}
			I & O \\
			tM_2 & I
		\end{bmatrix}
		\begin{bmatrix}
			I & O \\
			O & I-t^2M_2M_1
		\end{bmatrix}
		\begin{bmatrix}
			I & tM_1 \\
			O & I
		\end{bmatrix}
	$,
	$\det(I+tM) = \det(I-t^2M_2M_1)= \det(I-t^2\mu_1M_2')$ holds.
	Note that the matrix $-t\mu_1M_2'$ is an example of $M$ of Lemma \ref{lem:I+tJ(u,u)},
	and we get 
	$$
		\det(I-t^2\mu_1M_2') = \det(I+t(-t\mu_1M_2')) = 1+ (-\mu_1\mu_2 t)t = 1-\mu_1\mu_2 t^2.
	$$
	
\end{proof}

%%%%%%%%%%%%%%%%%%%%%%%%%%%%%%%%%%%%%%%%%%%%%%

%%%%%%%%%%%%%%%%%%%%%%%%%%%%%%%%%%%%%%%%%%%%%%
\subsection{The Ihara expression of the graph zeta function on a finite digraph}
\noindent
%%%%%%%%%%%%%%%%%%%%%%%%%%%%%%%%%%%%%%%%%%%%%%

%%%%%%%%%%%%%%%%%%%%%%%%%%%%%%%%%%%%%%%%%%%%%%

We fix a total order $<$ on $V$, and if one writes $\arc(u,v)$, then the condition $u<v$ is always assumed.
Let
$J=(j_{aa'})_{a,a'\in\arc}$,
$K=(k_{av})_{a\in\arc ,v\in V}$ and
$L=(l_{ua'})_{u\in V, a'\in\arc}$
be matrices defined by 
$j_{aa'}=\upsilon(a,a')\delta_{a'\in a\inv}$,
$k_{av}=\tau_1(a)\delta_{\head(a)v}$ and
$l_{ua'}=\tau_2(a')\delta_{u\tail(a')}$.
Recall that $\arc = \sqcup_{(u,v)\in\Phi_\Delta} \sqcup_{a\in B(u,v)} \arc(a)$ holds.
For each $a\in B(u,v)$ of $(u,v)\in\Phi_\Delta$,
let 
%$J(u,v) := (j_{aa'})_{a,a'\in\arc(u,v)}$,
%$K(u,v) := (k_{aw})_{a\in\arc(u,v),w\in V}$ and  
%$L(u,v) := (l_{wa'})_{w\in V,a'\in\arc(u,v)}$.
$J(\arc(a)) := (j_{aa'})_{a,a'\in\arc(a)}$,
$K(\arc(a)) := (k_{aw})_{a\in\arc(a),w\in V}$ and
$L(\arc(a)) := (l_{wa'})_{w\in V,a'\in\arc(a)}$.
Note that we can choose a total order on $V$, which makes $J$ a diagonal matrix.
We fix such a total order on $V$.
Let $T$ denote the block diagonal matrix $I+tJ$, 
and the diagonal blocks are given by $T(\arc(a)) := I+tJ(\arc(a))$.
Then, one can see that $\det T = \prod_{(u,v)\in\Phi_\Delta}\prod_{a\in B(u,v)} \det T(\arc(a))$ holds.

For a set of arcs $S$,
let $\upsilon(S):= \sum_{a\in S}\upsilon(a,a)$.
We consider the following matrices
\begin{align*}
	A_\Delta(\theta) &= \sum_{(u,v)\in \Phi_\Delta} \sum_{a\in B(u,v)} L(\arc(a)) K(\arc(a)), \\
	D_\Delta(\theta) &= \sum_{(u,v)\in \Phi_\Delta} \sum_{a\in B(u,v)} \frac{L(\arc(a)) J(\arc(a)) K(\arc(a))}{1-\upsilon(\arc[a])\upsilon(a\inv)t^2}, \\
	X_\Delta(\theta) &= \sum_{(u,v)\in \Phi_\Delta} \sum_{a\in B(u,v)} \frac{L(\arc(a)) J(\arc(a))^2 K(\arc(a))}{1-\upsilon(\arc[a])\upsilon(a\inv)t^2}.
\end{align*}

\begin{theorem}
	For a finite digraph $\Delta$,
	the following identity holds:
	$$
		Z_\Delta(t;\theta)\inv = \det(I-tM_\theta)
		= \det T  \det(I-tA_\Delta(\theta) +t^2 D_\Delta(\theta) -t^3X_\Delta(\theta)).
	$$
\end{theorem}

\begin{proof}
Let $H := (\tau(a,a')\delta_{\head(a)\tail(a')})_{a,a'\in\arc}$, and we have $M_\theta = H-J$.
Since $\tau(a,a')\delta_{\head(a)\tail(a')}=\sum_{v\in V}(\tau_1(a)\delta_{\head(a)v})(\tau_2(a')\delta_{v\tail(a')})$,
it follows that $H=KL$ and
$$
	\det(I-tM_\theta) =\det(I-t(KL-J)) = \det(T-tKL).
$$
%For two conformable matrices $X$ and $Y$, 
%it is known that $\det (I-XY) =\det (I-YX)$ holds.
Hence we have 
\begin{align*}
	\det(T-tKL) &= \det(T)\det(I-tT\inv KL) \\
	&=\det(T) \det (I-tLT\inv K).
\end{align*}
For the direct sum decomposition $T=\bigoplus_{\substack{(u,v)\in \Phi_\Delta, \\ a\in B(u,v)}}T(\arc(a))$,
we arrange the submatrices $L(\arc(a))$ and $K(\arc(a))$ of $L$ and $K$ in order of submatrices $T(\arc(a))$ of $T$.
Then, $LT\inv K$ is written by
\begin{eqnarray}
	\sum_{(u,v)\in \Phi_\Delta}\sum_{a\in B(u,v)} L(\arc(a))T(\arc(a))\inv K(\arc(a)). \label{eq:LTK}
\end{eqnarray}
%The matrix $T(\arc(a))$ is different for a loop $a$ and a non-loop $a$.
%Thus we consider these two cases separately

Recall that 
$$
\arc(a)=\arc[a]\cup a\inv=\{ a \}\sqcup \{\overline{a} \} 
$$ for $a\in\arc(G)$ on $\Delta(G)$,
and 
$$
	\arc(a)=\arc_{uv}\sqcup\arc_{vu}, \quad \arc(l)=\arc_{ww}\cup\arc_{ww}=\arc_{ww}
$$ for $a\in\arc_{uv}, l\in\arc_{ww}$ on $\Delta$ with $\Delta\neq\Delta(G)$.
For $a\in\arc(G)$ on $\Delta(G)$,
 $J(\arc(a))$ is given by
$$
	J(\arc(a))=\begin{bmatrix} 
		0 & \upsilon(a,\overline{a}) \\
		\upsilon(\overline{a},a) & 0 
	\end{bmatrix}
	=\begin{bmatrix} 
		0 & \upsilon_1(a)\upsilon_2(\overline{a}) \\
		\upsilon_1(\overline{a})\upsilon_2(a) & 0 
	\end{bmatrix},
$$
and from Lemma \ref{lem:I+tJ(u,v)}, the following holds:
\begin{align*}
	T(\arc(a))\inv &= I-t\frac{J(\arc(a)) }{1-\upsilon(a,a)\upsilon(\overline{a},\overline{a})t^2}
		+t^2\frac{J(\arc(a))^2}{1-\upsilon(a,a)\upsilon(\overline{a},\overline{a})t^2} \\
		&= I-t\frac{J(\arc(a)) }{1-\upsilon(\arc[a])\upsilon(a\inv)t^2}
		+t^2\frac{J(\arc(a))^2}{1-\upsilon(\arc[a])\upsilon(a\inv)t^2}.
\end{align*}
On the other hand, for 
%$a\in\arc_{uv}, l\in\arc_{ww}$ on 
$\Delta$,
let $J_{uv}:=(\upsilon(a',a''))_{a'\in\arc_{uv},a''\in\arc_{vu}}$. % and $J_{ww}:=(\upsilon(a',a''))_{a',a''\in\arc_{ww}}$
Then, 
$J(\arc(a))$ is given by
$$
	J(\arc(a))=\begin{cases}
		J_{uu} & {\rm if } \ a\in\arc_{uu}, \\
		\begin{bmatrix} 
			0 & J_{uv} \\
			J_{vu} & 0 
		\end{bmatrix} & {\rm otherwise}.
	\end{cases}
$$
In both cases,
from Lemma \ref{lem:I+tJ(u,u)} and Lemma \ref{lem:I+tJ(u,v)}, 
we can write $T(\arc(a))\inv$ as follows:
$$
	T(\arc(a))\inv = I -t\frac{J(\arc(a)) }{1-\upsilon(\arc[a])\upsilon(a\inv)t^2} +t^2\frac{J(\arc(a))^2}{1-\upsilon(\arc[a])\upsilon(a\inv)t^2}.
$$
Regardless of whether $\Delta$ is symmetric or not,
%That is, regardless of $\Delta$ is symmetric or not,
$L(\arc(a))T(\arc(a))\inv K(\arc(a))$ is written by
$$
	L(\arc(a))K(\arc(a)) -t\frac{L(\arc(a))J(\arc(a))K(\arc(a)) }{1-\upsilon(\arc[a])\upsilon(a\inv)t^2} +
		t^2\frac{L(\arc(a))J(\arc(a))^2K(\arc(a))}{1-\upsilon(\arc[a])\upsilon(a\inv)t^2}.
$$
Therefore, from (\ref{eq:LTK}), we have
\begin{align*}
	\det(I-tM_\theta) &= \det(T)\det(I-t(A_\Delta(\theta)-tD_\Delta(\theta)+t^2X_\Delta(\theta))) \\
	&= \det(T)\det(I-tA_\Delta(\theta)+t^2D_\Delta(\theta)-t^3X_\Delta(\theta)).
\end{align*}

\end{proof}

%\begin{remark}
%	The $(u,v)$-entries of $A_\Delta(\theta)$, $D_\Delta(\theta)$, and $X_\Delta(\theta)$ are as follows:
%	\begin{align*}
		
%	\end{align*}

%\end{remark}

%%%%%%%%%%%%%%%%%%%%%%%%%%%%%%%%%%%%%%%%%%
%%%%%%%%%%%%%%%%%%%%%%%%%%%%%%%%%%%%%%%%%%
%%%%%%%%%%%%%%%%%%%%%%%%%%%%%%%%%%%%%%%%%%
\section{Example} \label{sec:exam}

%\begin{figure}[h]
%\begin{minipage}[b]{0.49\columnwidth}
%\center
%\includegraphics[width=150pt]{digraph.png}
%\label{fig:digraph1}
%\caption{A digraph $\Delta=(V,\arc)$}
%\end{minipage}
%\begin{minipage}[b]{0.49\columnwidth}
%\center
%\includegraphics[width=150pt]{figure2.png}
%\label{fig:graph1}
%\caption{A graph $G=(V,E)$}
%\end{minipage}
%\end{figure}

We consider our zeta function on $\Delta=(V,\arc)$ in Figure \ref{fig:digraph1}.
%where $V=\{ v_1,v_2,v_3 \}$ and 
%$\arc=\{ a_{11}, a_{12}, a_{21}, a_{22}, a_{23}, a_{24}, a_{31}, a_{32}, a_{41}, a_{42} \}$.
%Let assume a total order $<$ on $V$ as $v_1<v_2<v_3$,　then $\Phi_\Delta =\{ (v_1,v_1), (v_1,v_2), (v_1,v_3), (v_2,v_3)\}$.
%On $\arc$, let $<$ $a_{ij}<a_{kl}$ when $i\leq k$ and $j< k$.
%\subsection{The zeta function for a digraph}
%The submatrix of $$
%then $\Phi_\Delta =\{ (v_1,v_1), (v_1,v_2), (v_1,v_3), (v_2,v_3)\}$.

We regard $\Delta$ as a usual graph as in Example \ref{exm:1}.
%We regard $\Delta$ as a usual graph not a symmetric digraph of any graph.
%Let 
%$$
%	B(v_1,v_1):=\{ a_{11} \}, B(v_1,v_2):=\{ a_{21} \}, B(v_1,v_3):=\{ a_{31} \}, B(v_2,v_3):=\{ a_{41} \}.
%$$
%By the definition, the set $\arc(a)$ for $a\in B(v_i,v_j)$ is given by $\arc(v_i,v_j)$.
For each $(v_i,v_j)\in\Phi_\Delta$ and $a\in B(v_i,v_j)$, $J(\arc(a))$, $K(\arc(a))$, and $L(\arc(a))$ are as follows:
\begin{align*}
	&J(\arc(a_{11}))=\begin{bmatrix}
		\upsilon(a_{11},a_{11})&\upsilon(a_{11},a_{12}) \\
		\upsilon(a_{12},a_{11})&\upsilon(a_{12},a_{12})
	\end{bmatrix},\\
	&K(\arc(a_{11}))=\begin{bmatrix}
		\tau_1(a_{11})&0&0 \\
		\tau_1(a_{12})&0&0
	\end{bmatrix} ,\quad 
	L(\arc(a_{11}))= \begin{bmatrix}
 	\tau_2(a_{11})&\tau_2(a_{12})\\
		0&0 \\
		0&0
 \end{bmatrix},\\
	&J(\arc(a_{21}))=\begin{bmatrix}	0&0&\upsilon(a_{21},a_{23})&\upsilon(a_{21},a_{24}) \\
		0&0&\upsilon(a_{22},a_{23})&\upsilon(a_{22},a_{24}) \\
		\upsilon(a_{23},a_{21})&\upsilon(a_{23},a_{22})&0&0 \\
		\upsilon(a_{24},a_{21})&\upsilon(a_{24},a_{22})&0&0
	\end{bmatrix},\\
	&K(\arc(a_{21}))=\begin{bmatrix}
		0&\tau_1(a_{21})&0 \\
		0&\tau_1(a_{22})&0 \\
		\tau_1(a_{23})&0&0 \\
		\tau_1(a_{24})&0&0
		\end{bmatrix},\\
	&L(\arc(a_{21}))=\begin{bmatrix}
 		\tau_2(a_{21})&\tau_2(a_{22})&0\\
		0&0&\tau_2(a_{23})&\tau_2(a_{24})\\
		0&0&0&0
		\end{bmatrix},\\
	&J(\arc(a_{31}))=\begin{bmatrix}
		0&\upsilon(a_{31},a_{32}) \\
		\upsilon(a_{32},a_{31})&0
		\end{bmatrix} ,\\
	&K(\arc(a_{31}))= \begin{bmatrix}
	 	0&0&\tau_1(a_{31}) \\
		0&\tau_1(a_{32})&0
		\end{bmatrix},\quad 
	L(\arc(a_{31}))= \begin{bmatrix}
 		0&0&0  \\
		\tau_2(a_{31})&0& 0  \\
		0&\tau_2(a_{32})&0
		\end{bmatrix},\\
	&J(\arc(a_{41}))=\begin{bmatrix}
		0&\upsilon(a_{41},a_{42}) \\
		\upsilon(a_{42},a_{41})&0 
	\end{bmatrix} ,\\
	&K(\arc(a_{41}))= \begin{bmatrix}
 		0&0&\tau_1(a_{41})  \\
		\tau_1(a_{42}) &0&0
	 \end{bmatrix},\quad 
	L(\arc(a_{41}))= \begin{bmatrix}
		\tau_2(a_{41})&0\\
		0&0\\
		0&\tau_2(a_{42}) 
	\end{bmatrix}.
\end{align*}
The matrices $A_\Delta(\theta), D_\Delta(\theta)$ and $X_\Delta(\theta)$ are given by
\begin{align*}
	&A_\Delta(\theta) = 
	\begin{bmatrix}
		\tau(a_{11},a_{11})+\tau(a_{12},a_{12}) & \tau(a_{21},a_{21})+\tau(a_{22},a_{22})&\tau(a_{41},a_{41}) \\
		\tau(a_{23},a_{23})+\tau(a_{24},a_{24})&0&\tau(a_{31},a_{31}) \\
		\tau(a_{42},a_{42})&\tau(a_{32},a_{32})&0
	\end{bmatrix}, \\
	&D_\Delta(\theta) = \begin{bmatrix}
		d_{v_1,v_1}+d_{v_1,v_2}+d_{v_1,v_3} &0&0 \\
		0&d_{v_2,v_1}+d_{v_2,v_3}&0 \\
		0&0&d_{v_3,v_1}+d_{v_3,v_2}
	\end{bmatrix}, \\ 
	&X_\Delta(\theta) = 
	\begin{bmatrix}
		0 & x_{v_1,v_2}& x_{v_1,v_3}\\
		x_{v_2,v_1}& 0 & x_{v_2,v_3}\\
		x_{v_3,v_1} & x_{v_3,v_2} & 0
	\end{bmatrix},
\end{align*}
where 
\begin{align*}
	&d_{v_i,v_j} = \left(1-t^2 \upsilon(\arc_{v_i,v_j})\upsilon(\arc_{v_j,v_i})\right)\inv 
		\left( \sum_{a\in\arc_{v_i,v_j},a'\in\arc_{v_j,v_i}} \upsilon(a,a')\tau(a',a) \right),\\
	&x_{v_i,v_j} = \left(1-t^2 \upsilon(\arc_{v_i,v_j})\upsilon(\arc_{v_j,v_i})\right)\inv 
		\left( \sum_{a,a'\in\arc_{v_i,v_j}}\sum_{a''\in\arc_{v_j,v_i}} \upsilon(a,a'')\upsilon(a'',a')\tau(a,a') \right).
\end{align*}

On the other hand, we regard $\Delta$ as a symmetric digraph of $G$ in Figure \ref{fig:graph1} as in Example \ref{exm:2}.
%Let $G=(V,E)$ be a graph with $V=\{ v_1,v_2,v_3 \}$ and 
%$E=\{ e_1=\{ v_1,v_1\}, e_2=\{ v_1,v_2\}, e_3=\{ v_1,v_2\}, e_4=\{ v_2,v_3\}, e_5=\{ v_1,v_3\} \}$.
%Assign 
%$e_1$ to $\{ a_{11},a_{12} \}$, 
%$e_2$ to $\{ a_{21},a_{23} \}$, 
%$e_3$ to $\{ a_{22},a_{24} \}$, 
%$e_4$ to $\{ a_{31},a_{32} \}$ and
%$e_5$ to $\{ a_{41},a_{42} \}$.
%For $(v_i,v_j)\in\Phi_{\Delta(G)}$, the set $B(v_i,v_j)$ is $\arc_{v_iv_j}$.
For $a\in B(v_i,v_j)$, the matrices $J(\arc(a)), K(\arc(a))$, and $L(\arc(a))$ are as follows:
\begin{align*}
	&J(\arc(a_{11}))=\begin{bmatrix}
		0 & \upsilon(a_{11},a_{12}) \\
		\upsilon(a_{12},a_{11}) & 0
	\end{bmatrix}, \quad 
	K(\arc(a_{11}))=\begin{bmatrix}
		\tau_1(a_{11})& 0 & 0\\
		\tau_1(a_{12})& 0 & 0
	\end{bmatrix}, \\
	&L(\arc(a_{11}))=\begin{bmatrix}
		\tau_2(a_{11})&\tau_2(a_{12}) \\
		0&0\\
		0&0
	\end{bmatrix}
,\\	
	&J(\arc(a_{21}))=\begin{bmatrix}
		0 & \upsilon(a_{21},a_{23}) \\
		\upsilon(a_{23},a_{21}) & 0
	\end{bmatrix},\quad 
	K(\arc(a_{21}))=\begin{bmatrix}
		0&\tau_1(a_{21})& 0\\
		\tau_1(a_{23})& 0 & 0
	\end{bmatrix}, \\
	&L(\arc(a_{21}))=\begin{bmatrix}
		\tau_2(a_{21})&0 \\
		 0 &\tau_2(a_{23}) \\
		0& 0 
	\end{bmatrix}
,\\
	&J(\arc(a_{22}))=\begin{bmatrix}
		0 & \upsilon(a_{22},a_{24}) \\
		\upsilon(a_{24},a_{22}) & 0
	\end{bmatrix},\quad 
	K(\arc(a_{22}))=\begin{bmatrix}
		0&\tau_1(a_{22})&0 \\
		\tau_1(a_{24})& 0 &0
	\end{bmatrix}, \\
	&L(\arc(a_{22}))=\begin{bmatrix}
		\tau_2(a_{22})& 0  \\
		0 &\tau_2(a_{24}) \\
		0 & 0 
	\end{bmatrix}
,\\	
	&J(\arc(a_{31}))=\begin{bmatrix}
		0 & \upsilon(a_{31},a_{32}) \\
		\upsilon(a_{32},a_{31}) & 0
	\end{bmatrix},\quad 
	K(\arc(a_{31}))=\begin{bmatrix}
		0& 0 &\tau_1(a_{31}) \\
		0&\tau_1(a_{32})&0
	\end{bmatrix},\\
	&L(\arc(a_{31}))=\begin{bmatrix}
		 0 & 0  \\
		\tau_2(a_{31})& 0  	\\
		 0 &\tau_2(a_{32})
	\end{bmatrix}
, \\	
	&J(\arc(a_{41}))=\begin{bmatrix}
		0 & \upsilon(a_{41},a_{42}) \\
		\upsilon(a_{42},a_{41}) & 0
	\end{bmatrix}, \quad 
	K(\arc(a_{41}))=\begin{bmatrix}
		0& 0 &\tau_1(a_{41})  \\
		\tau_1(a_{42}) & 0 &0
	\end{bmatrix}, \\
	&L(\arc(a_{41}))=\begin{bmatrix}
		\tau_2(a_{41})&0   \\
		0 & 0  \\
		0 &\tau_2(a_{42}) 
	\end{bmatrix}.
\end{align*}

\begin{comment}

%where the directions of $a_{11},\ldots, a_{42}$ are the same as in $\Delta$.
%Thus, the symmetric digraph $\Delta(G)$ is the same as $\Delta$.
%Note that two arcs corresponding to an edge are inverses of each other.
%and they do not have any other inverses.
%Let $\Delta(G)=(V,\arc(G))$ be the symmetric digraph of $G$
%with
%$\arc(G)=\{ a_{11}< a_{12} < a_{21}<a_{23}<a_{22}<a_{24}<a_{31}<a_{32} <a_{41}<a_{42} \}$.
The matrices $J,K,$ and $L$ are as follows:
\begin{eqnarray*}\small
	J=
	\begin{bmatrix}
		0&1&&&&&&&& \\
		1&0&&&&&&&& \\
		&&0&1&&&&&& \\
		&&1&0&&&&&& \\
		&&&&0&1&&&& \\
		&&&&1&0&&&& \\
		&&&&&&0&1&& \\
		&&&&&&1&0&& \\
		&&&&&&&&0&1 \\
		&&&&&&&&1&0 
	\end{bmatrix}, \ \
	K=
	\begin{bmatrix}
		\tau_1(a_{11})&& \\
		\tau_1(a_{12})&& \\
		&\tau_1(a_{21})& \\
		\tau_1(a_{23})&& \\
		&\tau_1(a_{22})& \\
		\tau_1(a_{24})&& \\
		&&\tau_1(a_{31}) \\
		&\tau_1(a_{32})& \\
		&&\tau_1(a_{41})  \\
		\tau_1(a_{42}) &&
	\end{bmatrix}, \\
	L=
	\begin{bmatrix}
		\tau_2(a_{11})&\tau_2(a_{12})&\tau_2(a_{21})&&\tau_2(a_{22})&&&&\tau_2(a_{41})&   \\
		&&&\tau_2(a_{23})&&\tau_2(a_{24})&\tau_2(a_{31})&&&   \\
		&&&&&&&\tau_2(a_{32})&&\tau_2(a_{42}) 
	\end{bmatrix}.
\end{eqnarray*}
\end{comment}
The matrices $A_{\Delta(G)}(\theta)$, $D_{\Delta(G)}(\theta)$, and $X_{\Delta(G)}(\theta)$ are given by
\begin{align*} 
	&A_{\Delta(G)}(\theta) = 
	\begin{bmatrix}
		\tau(a_{11},a_{11})+\tau(a_{12},a_{12}) & \tau(a_{21},a_{21})+\tau(a_{22},a_{22})&\tau(a_{41},a_{41}) \\
		\tau(a_{23},a_{23})+\tau(a_{24},a_{24})&0&\tau(a_{31},a_{31}) \\
		\tau(a_{42},a_{42})&\tau(a_{32},a_{32})&0
	\end{bmatrix}, \\
	&D_{\Delta(G)}(\theta) = \begin{bmatrix}
		d_{v_1,v_1}+d_{v_1,v_2}+d_{v_1,v_3} &0&0 \\
		0&d_{v_2,v_1}+d_{v_2,v_3}&0 \\
		0&0&d_{v_3,v_1}+d_{v_3,v_2}
	\end{bmatrix}, \\
	&X_{\Delta(G)}(\theta) = 
	\begin{bmatrix}
		0 & x_{v_1,v_2}& x_{v_1,v_3}\\
		x_{v_2,v_1}& 0 & x_{v_2,v_3}\\
		x_{v_3,v_1} & x_{v_3,v_2} & 0
	\end{bmatrix},
\end{align*}
where 
\begin{align*}
	d_{v_i,v_j} = \sum_{a\in\arc_{v_i,v_j}} \frac{ \upsilon(a,\overline{a})\tau(\overline{a},a)}{1-t^2\upsilon(a,a)\upsilon(\overline{a},\overline{a})}, \quad
	x_{v_i,v_j} = \sum_{a\in\arc_{v_i,v_j}} \frac{\upsilon(a,\overline{a})\upsilon(\overline{a},a)\tau(a,a)}{1-t^2\upsilon(a,a)\upsilon(\overline{a},\overline{a})}.
\end{align*}

\section*{Acknowledgements}
\noindent
The author is partially supported by Grant-in-Aid for JSPS Fellows (Grant No. JP20J20590).

%\section*{References}

\end{document}